\documentclass[10pt]{article}
\usepackage{geometry}                
\usepackage{amssymb}
\usepackage{amsmath}
\usepackage{amsthm}
\usepackage{graphicx}
\usepackage{amsthm}
\usepackage{color}
\usepackage{graphics}
\usepackage[british]{babel}
\usepackage{pinlabel}
\usepackage{enumerate}

\newtheorem{theorem}{Theorem}[section]

\newtheorem{cor}[theorem]{Corollary}
\newtheorem{lemma}[theorem]{Lemma}
\newtheorem{prop}[theorem]{Proposition}
\newtheorem{remark}[theorem]{Remark}
\theoremstyle{definition}
\newtheorem{definition}[theorem]{Definition}

\newtheorem*{claim}{Claim}

\newenvironment{proof_claim}{\paragraph{Proof of Claim:}}{\hfill$\star$ \vspace{1ex}}

\definecolor{magenta}{RGB}{203,0,150}
\definecolor{blueish}{RGB}{0,35,211}
\definecolor{greenish}{RGB}{0,100,20}

\newcommand{\h}{\hspace{2mm}}  

\usepackage[noindentafter]{titlesec}
 \title{Morse Geodesics in Torsion Groups}
 \date{\today}
 \author{Elisabeth Fink}

\begin{document}

\selectlanguage{british}

\maketitle

\begin{abstract}
In this paper we exhibit Morse geodesics, often called ``hyperbolic directions'', in infinite unbounded torsion groups. The groups studied are lacunary hyperbolic groups and constructed using graded small cancellation conditions. In all previously known examples, Morse geodesics were found in groups which also contained \emph{Morse elements}, infinite order elements whose cyclic subgroup gives a Morse quasi-geodesic. Our result presents the first example of a group which contains Morse geodesics but no Morse elements. In fact, we show that there is an isometrically embedded $7$-regular tree inside such groups where every infinite, simple path is a Morse geodesic. 
\end{abstract}

\section{Introduction}

A geodesic is called \emph{Morse}, if quasi-geodesics between points on it stay uniformly close. They are named after H.M.\;Morse and were introduced in \cite{morse_surfaceGenus}. The Morse Lemma for hyperbolic groups says that in a hyperbolic group every quasi-geodesic is Morse. The connection between Morse geodesics and asymptotic cones allowed C.\;Drutu and M.\;Sapir (\cite{drutu_sapir}) to conclude that relative hyperbolicity is preserved under quasi-isometries. Recently Morse geodesics have been shown to exist in acylindrically hyperbolic groups by A.\;Sisto (\cite{sisto_16}), a large class of groups introduced by D.\;Osin in \cite{osin_acylindrical} to formally unify several classes of groups. This generalized previous results about Morse geodesics in some classes of groups (\cite{behrstock}, \cite{AK}, \cite{hamenstadt_hensel}, \cite{KLP}). In \cite{cordes} the Morse boundary was introduced and further studied in \cite{cordes1}, \cite{cordes2}. Further, a classification of Morse geodesics via their divergence behavior has been studied in for example \cite{drutu_mozes_sapir}, \cite{ACGH}, \cite{tran}.

\medskip 

In any previously known example, it was possible to show the existence of Morse geodesics via infinite order elements, called \emph{Morse elements}. The embedding of the cyclic subgroup generated by such a Morse element is a Morse quasi-geodesic in the Cayley graph. In hyperbolic groups, any infinite order element is Morse. Morse elements must necessarily have infinite order. A clear obstruction to having Morse elements is when a group is a torsion group. Hence it is a natural question to ask, whether Morse geodesics can exist in groups without Morse elements. We answer this question affirmatively by exhibiting Morse geodesics in infinite, unbounded torsion groups coming from a graded small cancellation construction in \cite{osin_lacunary}.
In fact, we exhibit an entire isometrically embedded regular tree such that every infinite, simple path on it is a Morse geodesic. \begin{theorem}
Let $G$ be a group which satisfies a graded small cancellation condition, such that $G = \lim_{i \rightarrow \infty} G_i$ with $G_0 = F_4$. Then $G$ contains an isometrically embedded $7$-regular tree $T$, where every bi-infinite simple path is a Morse geodesic.
\end{theorem}
 The graded small cancellation condition on these groups gives the necessary technical tools for understanding the limit group. A result by C.\;Drutu and M.\;Sapir (\cite{drutu_sapir}) states that a group satisfying a non-trivial group law cannot have global cut-points in the asymptotic cone. Another result by the same authors together with S.\;Mozes (\cite{drutu_mozes_sapir}) connects cut-points in asymptotic cones with Morse geodesics. This implies that no group with bounded torsion can have Morse geodesics. Among other unbounded torsion groups, it is known that the Grigorchuk group (introduced in \cite{grigor_1}) cannot contain Morse geodesics because it is commensurable with its direct product (\cite{deLaHarpe}) and hence has no cut-points in any asymptotic cones.

\medskip

Our main construction builds upon the graded small cancellation construction used in \cite{osin_lacunary} to obtain  groups where at least one asymptotic cone is an $\mathbb{R}$-tree. Such groups are called \emph{lacunary hyperbolic groups}. It is shown in that paper, that among these groups, there are some which are  unbounded torsion groups. The idea is to take a direct limit of hyperbolic groups, each a quotient of the previous one by adding some relations, where the relations added satisfy a small cancellation criterion. This idea of taking such limits has been source of many interested examples of groups. Most notably Tarski-monsters  as introduced by A.\;Ol'shanskii in \cite{olshanskii_residualing}, which are group where any two elements generate the entire group. Over relatively hyperbolic groups, such an approach lead to infinite groups with just $2$ conjugacy classes as constructed by D.\;Osin (\cite{osin_smallCancellation}). Such techniques have been generalized by M.\;Hull to acylindrically hyperbolic groups (\cite{hull_acylin}). 

\medskip 

In this paper we treat groups with graded small cancellation which are a direct limit of hyperbolic groups. We show that there exists a convex subset which is a tree in the initial group that gets preserved in every step when taking a quotient  (Lemma~\ref{lemma_KstaysQuasiConvex_free}). We then conclude that every bi-infinite path on this tree is a Morse geodesic in the infinitely presented limit group (Theorem~\ref{thm:morse}). 

\medskip

While we are convinced that this construction works for any such sequence of hyperbolic groups, we will only show it in the case when the initial group is a free group with four generators. This reduction spares a lot of technical obstacles, which can be overcome rather easily but lengthen the proofs enormously. Due to the high level of technicality we have decided that the ``simple'' case where the first group in the sequence is $F_4$ is the best way to present our construction. The main idea for any general sequence of hyperbolic groups remains the same by exhibiting a free, quasi-convex subgroup within the initial group of the sequence. However this technicality turns a lot of segments which are geodesics in our proofs into quasi-geodesics, hence requiring much longer, more technical proofs.

\medskip

The key to our investigation is a result by C. Drutu and M. Sapir which relates Morse geodesics to asymptotic cones (\cite{drutu_mozes_sapir}). In \cite{osin_lacunary}, the authors show that relations in such graded small cancellation groups give rise to circles in asymptotic cones. Based on this, we aim to find paths in the Cayley graph which do not remain close to a circle coming from a relation for ``too long'' compared to the length of that relation. Then in the limit such a path will only touch a circle in the asymptotic cone, i.e.\;only have at most one point in common with a circle.

\medskip

The way to exhibit such paths is a combinatorial one. We first study the combinatorics of the set of relations in such graded small cancellation groups as words in $F_4$. 

\begin{lemma}
Let $\left<G | \mathcal{Q} \right>$ be a graded small cancellation presentation with sequences $\mu=\left(\mu_n\right), \rho=\left(\rho_n\right), \sigma = \left( \sigma_n \right)$ and $G_0=F_4$. Then the set $\mathcal{Q}$ contains at most $8^{\mu_n\sigma_n}$ relations of length $l$ for any length $\rho_n < l < \sigma_n$.
\end{lemma}

Once we have concluded that these relations are in fact not very ``dense'' in the free group $F_4$, we use an inductive counting argument to conclude that there are infinitely many words which have very little in common with such relations. In particular, it emerges from the proof of Lemma \ref{lemma_comb_free}, that such paths form a subtree of the Cayley graph of $F_4$. Even more, it emerges that it suffices to ``cut off'' one direction from every vertex in $Cay(F_4)$ and the remaining tree has the property that each path is a Morse geodesic in the graded small cancellation group. We would like to note, that if $G$ is an unbounded torsion group the pattern according to which we decide which direction in $Cay(F_4)$ will not be part of this tree must necessarily be non-periodic.

\medskip

\textbf{Acknowledgements:} This result has emerged from many inspiring discussions with Romain Tessera. It was only in the end that we decided not to write a joint paper. I have been invited to give a talk on this result in several places and would like to thank everyone for their interest and hospitality and for all the helpful comments that I have received. The ground work on this project has been done during my time as a postdoc of Anna Erschler at the ENS Paris where I was supported by the ERC starting grant 257110 ``RaWG''. I'm deeply grateful to both my current supervisors Vadim Kaimanovich and Kirill Zainoulline for allowing me to finish this project while supported by their respective grants. 

\section{Basic definitions}

In this section we define the necessary terms and notation. Familiarity with metric spaces, quasi-isometries, hyperbolicity and Cayley-graphs of groups is assumed. Further, we will reference the necessary ground work from \cite{osin_lacunary} that we build our construction on.

\medskip

\begin{definition}
A bi-infinite (quasi-)geodesic $g$ in a metric space is called a \emph{Morse} (quasi-)geodesic if for any $L>1, A\geq0$ there
exists $C=C(L,A)$ such that for any two points $x,y$ on $g$ and any $(L,A)$-quasi-geodesic $\gamma$ between $x$ and $y$, we
have that $\gamma$ stays $C$-close to $g$. If $g$ is the embedding of an infinite cyclic group $\left<x\right>$ into the Cayley
graph of $G$, then $x$ is called a \emph{Morse element}.
\end{definition}

\subsection{Asymptotic cones}

Asymptotic cones are limit spaces of Cayley graphs of groups. They were first introduced by M.\;Gromov (\cite{gromov_nil_poly}) and a description via ultrafilters was then given by L.\;van~den~Dries and A.\;Wilkie (\cite{vdDries_Wilkie}).

\medskip

A \emph{non-principal ultrafilter} $\omega$ is a finitely additive measure defined on all subsets $S$ of $\mathbb{N}$, such
that $\omega(S) \in \left\{0,1\right\}$, $\omega(\mathbb{N})=1$ and $\omega(S)=0$ if $S$ is a finite subset. Given two infinite
sequences of real numbers $(a_n), (b_n)$, we write $a_n = o_\omega(b_n)$ if $\lim_\omega a_n/b_n = 0$. On the set of all infinite
sequences of real numbers we define an equivalence relation. We write $a_n = o_\omega(b_n)$ if $\lim_{\omega} a_n/b_n = 0$.
Similarly, we say $a_n=\Theta_{\omega}(b_n)$ if $0<\lim_{\omega}a_n/b_n < \infty$. Further we say $a_n = O_\omega(b_n)$ if
$\lim_\omega a_n/b_n < \infty$.

\medskip

Let $(X_n, dist_n), n \in
\mathbb{N}$ be a sequence of metric spaces. Fix an arbitrary sequence $e=(e_n)$ of points $e_n \in X_n$.
The $\omega$-limit $\lim_\omega (X_n)_e$ is the quotient space of equivalence classes where the distance between $(f_n)^\omega$
and $(g_n)^\omega$ is defined as $\lim_\omega dist(f_n,g_n)$. An \emph{asymptotic cone} $Con^{\omega} (X,e,d)$ of a metric
space $(X, dist)$ where $e=(e_n)$, $e_n \in X$, and $d=(d_n)$ is
an unbounded non-decreasing scaling sequence of positive real numbers, is the $\omega$-limit of spaces $X_n=(X,dist/d_n)$.

\medskip

\begin{definition}
Let $F$ be a complete geodesic metric space and let $\mathcal{P}$ be a collection of closed geodesic non-empty subsets (called
\emph{pieces}). Suppose that the following two properties are satisfied:

\begin{enumerate}
\item Every two different pieces have at most one common point.
\item Every non-trivial simple geodesic triangle (a simple loop composed of three geodesics) in $F$ is contained in one piece.
\end{enumerate}
Then we say that the space $F$ is \emph{tree-graded with respect to $\mathcal{P}$}. 

\medskip

The topological arcs starting in a given point and intersecting each piece in at most one point compose a real tree called
\emph{transversal tree}. 
\end{definition}

\medskip

The following proposition from \cite{drutu_mozes_sapir} describes Morse geodesics in terms of asymptotic cones which are tree-graded structures. We give two out
of five equivalences here.

\begin{prop}[\cite{drutu_mozes_sapir}]\label{prop_transversal}
Let $X$ be a metric space and let $q$ be a bi-infinite quasi-geodesic in $X$ and for every two points $x,y$ on $q$.
The following conditions are equivalent for $q$:
\begin{enumerate}
\item In every asymptotic cone of $X$, the ultralimit of $q$ is either empty or contained in a transversal tree for some
tree-graded structure;
\item $q$ is a Morse quasi-geodesic;
\end{enumerate}
\end{prop}

In particular, it emerges from the proof of the above proposition that the limit of a Morse quasi-geodesic $q$ in every asymptotic
cone is such that 
\begin{enumerate}
\item every point of $q$ is a cut-point,
\item every point of $q$ separates $q$ into two halves which lie in two different connected components.
\end{enumerate}

\subsection{Graded small cancellation}

The following definitions are a summary of the basic definitions in \cite{osin_lacunary}, some of which can originally be found in \cite{olshanskii_residualing}. The main idea here is that if we have a 'nice' group $H$ and we take an appropriate quotient $H_1$ by imposing strong conditions on the relations that we add, then $H_1$ will again have 'nice' properties. Now take a sequence of groups $\{H_i\}$ where in each $H_i$ the relations satisfy certain conditions. By introducing asymptotic conditions on the conditions in each group $H_i$, we can assure that the limit of the $H_i$ will again be an infinite group. Further, these conditions allow the relations to be chosen in such a way that we will obtain infinite groups with desired properties. Among such infinite groups resulting from this are infinite unbounded torsion groups. 

\medskip

To study small cancellation conditions, we first introduce an $\epsilon$-piece. The main difference to classical small cancellation is that it is enough for two relations to be 'relatively close to each other' to be said to have a common piece. The pieces do not necessarily need to be a joint segment of the relations, just two segments which are close to each other. 

\begin{definition}
Let $H$ be a group generated by a set $S$. Let $\mathcal{R}$ be a symmetrized set of reduced words in $S^{\pm 1}$. For $\epsilon
>0$, a subword $U$ of a word $R \in \mathcal{R}$ is called an \emph{$\epsilon$-piece} if there exists a word $R' \in \mathcal{R}$
such that:

\begin{enumerate}
\item $R \equiv UV, R' \equiv U'V'$, for some $V, U', V'$;
\item $U' \equiv YUZ$ in $H$ for some words $Y,Z$ such that $\max\{|Y|, |Z|\} \leq \epsilon$;
\item $YRY^{-1} \neq R'$ in the group $H$.
\end{enumerate}
\end{definition}

Recall that a word $W$ in the alphabet $S^{\pm 1}$ is called $(\lambda, c)$-quasi-geodesic (respectively geodesic) in $H$ if any
path in $\Gamma(H,S)$ labeled by $W$ is $(\lambda, c)$-quasi-geodesic (respectively geodesic).

\begin{definition}
Let $\epsilon \geq 0, \mu \in (0,1)$ and $\rho>0$. We say that a symmetrized set $\mathcal{R}$ of words over the alphabet $S^{\pm
1}$ satisfies the \emph{condition $C(\epsilon, \mu, \rho)$} for the group $H$, if
\begin{enumerate}
\item all words from $\mathcal{R}$ are geodesic in $H$;
\item $|R| \geq \rho$ for any $R \in \mathcal{R}$;
\item the length of any $\epsilon$-piece contained in any word $R \in \mathcal{R}$ is smaller than $\mu \cdot |R|$. 
\end{enumerate}
\end{definition}

We use the above definitions to define graded small cancellation groups.

\begin{definition}\label{def_graded_sc_gps}
Let $\alpha, K$ be positive constants. We say that the presentation
\begin{equation}\label{eq_gsc_pres}\left<S \mid \mathcal{R}\right> = \left<S \mid \bigcup_{i=0}^{\infty}
\mathcal{R}_i\right>\end{equation} of a group $G$ is a $\mathcal{Q}(\alpha, K)$-presentation if the following conditions hold for
some sequences $\epsilon=(\epsilon_n), \mu=(\mu_n)$ and $\rho=(\rho_n)$ of positive real numbers ($n=1,2 \dots$), where $\epsilon_n \geq 1$ for all $n$. 

\begin{enumerate}
\item[$(Q_0)$]\label{part_gsc_q_0} The group $G_0= \left<S \mid \mathcal{R}_0\right>$ is $\delta_0$-hyperbolic for some $\delta_0$.
\item[$(Q_1)$]\label{part_gsc_q_1} For every $n \geq 1$, $\mathcal{R}_n$ satisfies a $C(\epsilon_n, \mu_n, \rho_n)$ condition over
$G_{n-1}=\left<S \mid
\bigcup_{i=0}^{n-1} \mathcal{R}_i\right>$.
\item[$(Q_2)$]\label{part_gsc_q_2} $\mu_n=o(1), \mu_n \leq \alpha$ and $\mu_n\rho_n > K \epsilon_n $ for any $n \geq 1$.
\item[$(Q_3)$]\label{part_gsc_q_3} $\epsilon_{n+1} > 8 \cdot \max\{|R|, R \in \mathcal{R}_n\} = O(\rho_n)$.
\end{enumerate}
\end{definition}

The following lemma shows that if $\alpha$ is small enough and $K$ is big enough, groups with
$\mathcal{Q}(\alpha,K)$-presentations have nice properties.

\begin{lemma}\label{lem:graded}
Let \eqref{eq_gsc_pres} be a $\mathcal{Q}(0.01, 10^6)$-presentation. Then the following conditions hold:
\begin{enumerate}[(a)]
 \item For every $n \geq 1$, $G_n = \langle H \mid \mathcal{R}_n \rangle$ is $\delta_n$-hyperbolic with $\delta_n \leq 4 \cdot \max \{|R| \mid R \in \mathcal{R}\}$.
 \item $\epsilon_n = o(\rho_n)$
 \item $\rho_n = o(\rho_{n+1})$, in particular $\rho_n \rightarrow \infty$ as $n \rightarrow \infty$ and $\rho_n = o(\rho_{n+1})$.
 \item $\rho_n = o(r_S(G_n\rightarrow G_{n+1})$, where $r_S$ is the injectivity radius.
\end{enumerate}
\end{lemma}

\begin{definition} From now on the condition $\mathcal{Q}=\mathcal{Q}(0.01,10^6)$ will be called the
\emph{graded small cancellation condition} and a group satisfying it a \emph{graded small cancellation group}.\end{definition}

In addition to the above sequences, we define for a graded small cancellation group $\sigma_n = \max\{|R| \mid R \in \mathcal{R}_n\}$. The next proposition summarizes some inequalities that follow from the above definitions that will be frequently used in the next section. 

\begin{prop}\label{prop_inequalities}
Let $G$ be a graded small cancellation group. Then the following inequalities hold for the sequences $\{\epsilon_i\}, \{\rho_i\}, \{\mu_i\}$ for all $i \in \mathbb{N}$: 
\begin{enumerate}[(i)]
 \item \label{inequ1} $10^{-8} > \epsilon_i/\rho_i$,
 \item \label{inequ2} $\rho_{i+1} > 10^8 \cdot 8 \cdot \max \{|R| \mid R \in \mathcal{R}_{i}\} = 10^8 \cdot 8 \cdot \sigma_i $
 \item \label{inequ3} $\mu_n \rho_n > K \epsilon_n$
 \item \label{inequ4} $K = 10^6, \mu_i \leq 0.01$
 \item \label{inequ5} $\lim_{n \rightarrow \infty} \mu_n = 0$. 
 \item $\mu_n\rho_n> K \epsilon_n> K$
 \item $\sigma_n > K/\mu_n$
\end{enumerate}

\end{prop}

\begin{proof}
All of these follow immediately by combining the above conditions.
\end{proof}

Figure \ref{fig:relations} visualizes the distribution of lengths of relations. Each set $\mathcal{R}_i$ contains relations of lengths between $\rho_i$ and $\sigma_i$. The proportion of each relation that can be an $\epsilon_i$-piece becomes smaller and smaller. 

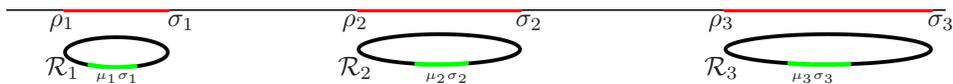
\begin{figure}[h]
\begin{center}
\labellist
\pinlabel $\mathcal{R}_1$ at 30 0
\pinlabel $\mathcal{R}_2$ at 190 0
\pinlabel $\mathcal{R}_3$ at 390 0
\pinlabel $\rho_1$ at 30 25
\pinlabel $\sigma_1$ at 95 25
\pinlabel $\rho_2$ at 190 25
\pinlabel $\sigma_2$ at 285 25
\pinlabel $\rho_3$ at 390 25
\pinlabel $\sigma_3$ at 510 25
\pinlabel \tiny{$\mu_1 \sigma_1$} at 60 -5
\pinlabel \tiny{$\mu_2 \sigma_2$} at 240 -5
\pinlabel \tiny{$\mu_3 \sigma_3$} at 440 -5
\endlabellist
\includegraphics[scale=0.7]{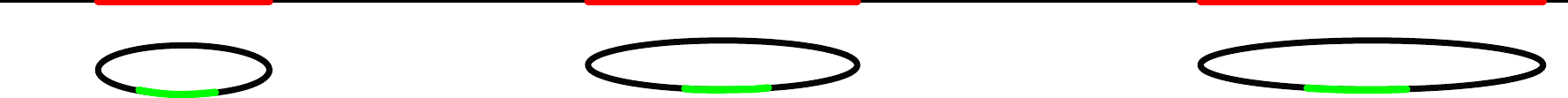} \caption{\small{Depicting the lengths of relations in the sets $\mathcal{R}_i$. We have $\rho_i > 10^8 \cdot \sigma_i$. The proportions of length less than $\mu_i\sigma_i$  of the relations which can be $\epsilon_i$-pieces become smaller and smaller. }\label{fig:relations}}
\end{center}
\end{figure}

\subsection{Lacunary Hyperbolic Groups}

\begin{definition}
A group is called \emph{lacunary hyperbolic} if one of its asymptotic cones is an $\mathbb{R}$-tree.
\end{definition}

These have been characterized in \cite{osin_lacunary} as limits of hyperbolic groups. Given a group homomorphism
$\alpha:G\rightarrow H$ and a generating set $S$ of $G$, we denote by $r_S(\alpha)$ the \emph{injectivity radius of $\alpha$ with
respect to $S$}, i.e. the radius of the largest ball $B$ in $G$ such that $\alpha$ is injective.

\begin{theorem}[\cite{osin_lacunary}]
Let $G$ be a finitely generated group. Then the following conditions are equivalent.

\begin{enumerate}
\item $G$ is lacunary hyperbolic.
\item There exists a scaling sequence $d=(d_n)$ such that $Con^\omega(G,d)$ is an $\mathbb{R}$-tree for any non-principal
ultrafilter $\omega$.
\item $G$ is the direct limit of a sequence of finitely generated groups and epimorphisms 
\begin{equation}\label{eq_lac} G_1 \xrightarrow{\alpha_1} G_2 \xrightarrow{\alpha_2} \dots \end{equation} such that $G_i$ is
generated by a finite set $S_i$, $\alpha_i(S_i)=S_{i+1}$, and each $G_i$ is $\delta_i$-hyperbolic, where
$\delta_i=o(r_{S_i}(\alpha_i))$.
\end{enumerate}
\end{theorem}

In \cite{osin_lacunary}, A.Yu.\;Olshanskii, D.\;Osin and M.\;Sapir have intensively studied such groups. In particular, the case
when all the groups in \eqref{eq_lac} are graded small cancellation groups. In this case, their asymptotic cones are rather well
understood.

\begin{definition}
We say that a metric space $X$ is a \emph{circle-tree}, if $X$ is tree graded with respect to a collection of circles whose radii
are uniformly bounded from below and above by positive constants.
\end{definition}

\begin{definition}
Given an ultrafilter $\omega$ and a scaling sequence $d=(d_n)$, we say that a sequence of real numbers $f=(f_n)$ is
$(\omega,d)$-visible if there exists a subsequence $(f_{n_i})$ of $f$ such that $f_{n_i}=\Theta_{\omega}(d_i)$.
\end{definition}

\begin{theorem}[\cite{osin_lacunary}]\label{thm_gsc_ac}
For any group $C$ having a graded small cancellation presentation, any ultrafilter $\omega$ and any sequence of constants
$d=(d_n)$, the asymptotic cone $Con^{\omega}(G,d)$ is a circle-tree. $Con^{\omega}(G,d)$ is an $\mathbb{R}$-tree if and only if
the sequence $(\rho_n)$ from Definition \ref{def_graded_sc_gps} is not $(\omega, d)$-visible.
\end{theorem}

The next lemma states that geodesic segments in the Cayley graph of a group $G$, which represent relations in a quotient $H$ of $G$, remain geodesic in the Cayley graph of $H$ on almost half their length.

\begin{lemma}[\cite{osin_lacunary}]\label{lemma_44Lacunary}
Suppose that the group $H$ is hyperbolic. Let $\mathcal{R}$ be a set of geodesic words in $H$, $\Delta$ a diagram over a
presentation $$H_1 = \langle H \mid \mathcal{R}\rangle = \langle S \mid \mathcal{O} \cup \mathcal{R}\rangle,$$ and $q$ a subpath of
$\partial \Delta$ whose label is geodesic in $H_1$. Then for any $\epsilon \geq 0$, no $\mathcal{R}$-cell $\Pi$ in $\Delta$ can
have an $\epsilon$-contiguity subdiagram $\Gamma$ to $q$ such that $(\Pi, \Gamma, q) > 1/2 + 2\epsilon/|\partial \Pi|$.
\end{lemma}


The following Lemma constitutes the main technical tool for our proofs. It is the key observation that a $C(\epsilon, \mu, \rho)$ quotient of a hyperbolic group is again hyperbolic. More detailed, it gives a geometric description of the geodesic segments in the Cayley graph of the quotient.

\begin{lemma}[\cite{osin_lacunary}]\label{lemma_46Lacunary}
Suppose that $H$ is a $\delta$-hyperbolic group having presentation $\left<S|\mathcal{O}\right>, \epsilon
 > 2 \delta, 0 < \mu \leq 0.01$, and $\rho$ is large enough (it suffices to choose $\rho > 10^6 \epsilon /\mu$). Let $H_1$ be
is a finite symmetrized set of words in $S^{\pm 1}$ satisfying the $C(\epsilon, \mu, \rho)$-condition. Then the following
statements hold.

\begin{enumerate}
\item Let $\Delta$ be a minimal disc diagram over $H_1=\left<S \mid \mathcal{O} \cup \mathcal{R}\right>$. Suppose that $\partial
\Delta = q^1\dots q^t$, where the labels of $q^1, \dots, q^t$ are geodesic in $H$ and $t \leq 12$. Then, provided $\Delta$ has an
$\mathcal{R}$-cell, there exists an $\mathcal{R}$-cell $\Pi$ in $\Delta$ and disjoint $\epsilon$-contiguity subdiagrams $\Gamma_1,
\dots, \Gamma_t$ (some of them may be absent) of $\Pi$ to $q^1, \dots, q^t$ respectively such that 
\[(\Pi, \Gamma_1, q^1) + \cdots + (\Pi, \Gamma_t, q^t) > 1 - 23 \mu.\]
\item $H_1$ is a $\delta_1$-hyperbolic group with $\delta_1 \leq 4r$ where $r = \max\{|R| \mid R \in \mathcal{R}\}$.
\end{enumerate}
\end{lemma}

Part one of the above Lemma states that the sides of a geodesic triangles in the quotient are $\epsilon$-close to an added relation for a large part.

\section{Morse Geodesics}

In this section we show that a lacunary hyperbolic group which are graded small cancellation groups starting with the free group $F_4$ contains Morse geodesics. The asymptotic cones of such groups are circle trees. The main idea of our construction is to show that there are geodesic rays such that their limit in the asymptotic cone is a transversal, i.e.\;only touches any circles in at most one point. These circles are limits of loops in the Cayley-graph which come from the relations in the group. Hence we aim to find paths, such that they only 'fellow-travel' with relations for a short amount of time. Depending on the length of the relation, we show that this common proportion becomes smaller and smaller, hence in the limit in the asymptotic cone, the path will only touch the circle. The proof is divided into three main sections.

\medskip

In the first subsection we prove some combinatorial observations regarding the structure of relations in such graded small cancellation groups. In particular, we count how many different subwords of a certain length there are. Then we show that in the free group $F_4$ there are many other infinite words which do not contain subwords of the relations of at least a certain length. We conclude that these valid words from a $7$-regular tree in the Cayley-graph of $F_4$.

\medskip

In the second subsection, we show that this tree of valid words remains a convexly embedded tree in a quotient formed under the graded small cancellation restrictions. We give structure theorems about triangles attached to this tree, which will allow us in the next subsection to conclude that there are no simple geodesic hexagons in any group $G_i$ with one side on this tree.

\medskip

In the last subsection, we conclude that the geometry of this embedded tree is such, that in the limit group every bi-infinite path on it will be a Morse geodesic. This then gives us the conclusion that there are infinite, unbounded torsion groups which contain Morse geodesics but no Morse elements.

\medskip

Let for the rest of this section $G$ be a lacunary hyperbolic group that satisfies a graded small cancellation condition. The group $G$ will be the limit of groups $G_i$ where $G_0 = F_4$, the free group in four generators.

\subsection{Combinatorial observations}

We discuss the combinatorics of relations in graded small cancellation groups. We will conclude that the set of words in $F_4$ which will become relations in $G$ is very 'sparse' in the Cayley graph of $F_4$. In particular, the small cancellation condition implies that the set of 'longer' pieces of such relations is very sparse as well. 

\begin{lemma}\label{lemma_sc_length}
Let $\left<G | \mathcal{Q} \right>$ be a graded small cancellation presentation with sequences $\mu=\left(\mu_n\right), \rho=\left(\rho_n\right)$ and $G_0=F_4$. Then the set $\mathcal{Q}$ contains at most $8^{\mu_n\sigma_n}$ relations of length $l$ for any length $\rho_n < l < \sigma_n$.
\end{lemma}

\begin{proof}
Let $R$ be a relation in $\mathcal{Q}$ such that $l=|R|$ satisfies $\sigma_n > l > \rho_n$ for some $n$. Define with this $n$ the set \[\mathcal{C}_n = \{ w \in F_4 \mid l(w) = \lfloor \mu_n\sigma_n \rfloor \}.\]
Considering the growth of the free group $F_4$, we get immediately that $|\mathcal{C}_n| \leq 8^{\mu_n\sigma_n}$. Then there exists $w_R \in \mathcal{C}_n$ such that $w_R$ is a subword of $R$. The small cancellation condition now requires that there does not exist another $R_1 \in \mathcal{Q}$ such that $w_R$ is a subword of $R_1$. Hence the maximal number of relations of length $\sigma_n > l > \rho_n$ is bounded by the number of different words of length $\mu_n \cdot \sigma_n$.
\end{proof}

\begin{remark}
  Lemma \ref{lemma_sc_length} gives a very coarse estimate. We exclude only relations which contain the exact same subword of a given length exceeding the parameters of small cancellation. The graded small cancellation condition actually forbids not just common subwords, but anything in an $\epsilon_n$-neighborhood of a given subword.
\end{remark}

\textbf{Notation:} Let $W,w$ be words in the free group $F_4$. If $w$ is a subword of $W$, we denote this by $w \subset W$.

\begin{lemma}\label{lemma_sc_numberPieces}
Let $\langle G \mid \mathcal{Q} \rangle$ be a graded small cancellation presentation. Denote by 
\[\mathcal{B}(\sqrt{\mu_n}\sigma_n) = \{w \mid l_{F_4}(w) = \sqrt{\mu_n}\sigma_n, \h \exists R \in \mathcal{Q}: \rho_n < |R| < \sigma_n \mbox{ and } w \subset R\}\]  the set of all subwords $w$ of relations $R$ in $\mathcal{Q}$ of length $|w| = \sqrt{\mu_n}\sigma_n$ if $\rho_n < |R| < \sigma_n$. Then  \[| \mathcal{B}( \sqrt{\mu_n} \sigma_n)| < \sigma_n \cdot 8^{\mu_n\sigma_n}.\] In particular, if $k=\sqrt{\mu_n}\sigma_n$, then $|\mathcal{B}(k)| \leq 8^{k/2}.$
\end{lemma}

\begin{proof}Making a coarse assumption that a new subword starts at every point of a relation $R$, we get that $R$ has at most $\sigma_n \geq |R|$ subwords of length $\sqrt{\mu_n}\sigma_n$. From Lemma \ref{lemma_sc_length} we get that we have at most $8^{\mu_n \sigma_n}$ relations. Hence we have at most $\sigma_n \cdot 8^{\mu_n\cdot \sigma_n}$ such disjoint pieces, which proves the first statement.

\medskip 

For the second part, we use $\mu_n\leq 1/100$ and $\sigma_n 8^{\mu_n\sigma_n}=8^{\log_8 \sigma_n + \mu_n \sigma_n}$ to prove:

\[8^{\log_8 \sigma_n + \mu_n \sigma_n}  \leq 8^{\sqrt{\mu_n}\sigma_n/2}.\] It is enough to show the inequality  is true if we replace $\log_8(\sigma_n)$ by $\ln(\sigma_n)$ as $8 >e$ and hence $\log_8(\sigma_n) < \ln(\sigma_n)$. This inequality then transforms into 
\[\mu_n \sigma_n + \ln \sigma_n \leq \frac{\mu_n^{\frac12} \sigma_n}2 \]
\begin{equation}2\mu_n + \frac{2 \ln \sigma_n}{\sigma_n} \leq \sqrt{\mu_n}.\label{eq_mu_ln_root}\end{equation}
By Proposition \ref{prop_inequalities} we have $\rho_n> 10^8$  and $\mu_n\rho_n> K \epsilon_n> K$ which yields that \begin{equation}\sigma_n>K/\mu_n\label{eq_rho_mu}\end{equation} where $K=10^6$. The function  $f(x) = 2 \ln(x)/x$ is descending for $x>e$. Hence we can put the inequality \eqref{eq_rho_mu} into  the function $f(x)$ to obtain  \[\frac{2\ln \sigma_n}{\sigma_n} < \frac{2 \ln\left(\frac{K}{\mu_n}\right)}{\frac{K}{\mu_n}}=\frac{\mu_n \cdot 2 \left(\ln(K) - \ln(\mu_n)\right)}{K}.\] Put into \eqref{eq_mu_ln_root} we get
\[2\mu_n\left(1+\frac{\ln(K)}{K} - \frac{\ln(\mu_n)}{K}\right) < \mu_n^{\frac12}\]
\begin{equation}2\left(1+\frac{\ln(K)}{K} + \frac{\ln \left(\frac{1}{\mu_n}\right)}{K} \right) < \mu_n^{-\frac12}.\label{eq_mu_1}\end{equation}
We set for readability $m_n=(\mu_n)^{-1}$. In particular, Proposition \ref{prop_inequalities} implies $m_n \geq 100$.  Rewritten, \eqref{eq_mu_1} becomes 
\[2(1+\ln(K)/K+\ln(m_n)/K)<m_n^{1/2}. \] Now $K=10^6$, so using again that $f(x)$ is descending we get that it is enough to check
\[2(1+\ln(10^6)/10^6+\ln(m_n)/10^6)<m_n^{1/2}\]   which is weaker than 
\[3+\frac{2 \ln (m_n)}{10^6} < \sqrt{m_n}\] which is true in particular if 
\[\frac{3}{\sqrt{m_n}} +\frac{2 \ln (m_n)}{10^6\cdot \sqrt{m_n}} < 1.\]
The function on the left side of the inequality is again descending for all $m_n > 1$ and has value less than $1$ for $m_n=100$. Hence $\mathcal{B}(k) \leq 8^{k/2}$.
\end{proof}

We explain the notation used for the rest of this section. We work with a lacunary hyperbolic group $G$, which is by \cite{osin_lacunary} a direct limit of hyperbolic groups $G_i$. Here we assume $G_0 = F_4$: 
\[G_0 = F_4 \twoheadrightarrow G_1 \twoheadrightarrow G_2 \twoheadrightarrow \dots\] where $F_4=\left<x_1,x_2,x_3,x_4\right>$ is a $4$-generated free group and $G_i$ is a $\delta_i$-hyperbolic group satisfying a $C(\epsilon_i, \mu_i, \rho_i)$-condition. We recall there is a one-to-one correspondence between reduced words in any group $G$ with generating set $S$ and geodesics in its Cayley-graph $Cay(G,S)$. Further, multiple reduced words (geodesic segments) can represent the same element of $G$. Hyperbolicity can be described as uniformly bounding the distance between two geodesic segments representing the same element (\cite{papasouglu}). 

\begin{lemma}\label{lemma_comb_free}
Let $\mathcal{N}$ be a countable set of words representing geodesic segments in the generators $x_1,\dots, x_4, x_1^{-1}, \dots, x_4^{-1}$ of free group $F_4$, sorted ascending by their length and further lexicographically. We assume \begin{enumerate}[(a)]
                                                                                                                                                       \item each $N_i \in \mathcal{N}$ is such that $n_i=l(N_i) \geq 10^4$,
                                                                                                                                                       \item \label{lemma_comb_free:cond2} there are at most $8^{n_i/2}$ segments of length $n_i$,
                                                                                                                                                       \item and $n_{i+1} - n_i >  10$ if $n_{i+1} \neq n_i$.
                                                                                                                                                      \end{enumerate}Then there is an isometrically embedded $7$-regular tree of infinite paths in $Cay(F_4)$ (respectively infinite words in $F_4$) with labels letters $x_1, \dots, x_4, x_1^{-1}, \dots, x_4^{-1}$ such that none of these paths fully contains any translate of any of the $N_i$ (respectively $n_i$ as a subword). 
\end{lemma}

 \begin{proof}
 Let $F(n)$ be the following subset of $F_4$:
 \[F(n) = \{ w \in F_4 \mid  l(w) = n, \h N_i \mbox{ is not a subword of } w\}.\]
 
 Denote the number of such words of length $n$ by $f(n)=|F(n)|$. We show $f(n)> 6 \cdot
 f(n-1)$ by induction. This is clear for $n_0=10^4$ as all words are allowed. Iterating the inequality leads to \begin{equation}f(n-i) < 6^{1-i}
 f(n-1).\label{eq_induction_free}\end{equation} 
 Now we extend the words of length $n-1$ in $F_4$ to words of length $n$ in $F_4$ with one of the $8$ letters $x_1,\dots, x_4, x_1^{-1}, \dots, x_4^{-1}$. Then the number of words of length $n$ of distance $>0$ from any $N_i$ can be computed as follows: take all possible reduced (as words in $F_4$) extensions ($7\cdot f(n-1)$) and subtract all words which would now contain any of the $N_i$ but previously did not. To count such words, we consider all words that did not contain any $N_i \in \mathcal{N}$ until length $n-n_i$. This is, all words that could possibly contain $n_i-1$ letters of any of the words $N_i$ and hence could be extended to some $N_i$ by adding one letter. This number is expressed as $\sum_{n_i < n} f(n-n_i)$. We get 
  \begin{equation} \label{eq_recursion}f(n) \geq 7 \cdot f(n-1) - \sum_{n_i < n}f(n-n_i).\end{equation}  We then use \eqref{eq_induction_free} to replace all $f(n-n_i)$ in the second term by $f(n-1)\cdot 6^{1-n_i}$. 

\[f(n) > 7 \cdot f(n-1) - \sum_{n_i<n} f(n-1) \cdot 6^{1-n_i}.\]  
  
We now use condition \eqref{lemma_comb_free:cond2} on $\mathcal{N}$ that there are at most $8^{n_i/2}$ words of length $n_i$ in $\mathcal{N}$:
  
  \begin{equation} \label{eq_recursion4}f(n) \geq 7 \cdot f(n-1) - f(n-1) \cdot \sum_{i <  n} 8^{i/2} \cdot 6^{1-i}\end{equation}

 It can be verified that $\sum_{i < n} 8^{i/2} \cdot 6^{1-i}=6  \cdot \sum_{i=10^4}^\infty \left(\frac{\sqrt{2}}{3}\right)^i  < 1$. Using this  we get \[f(n) > 6\cdot f(n-1). \] As a consequence, $f(n)$ is a strictly increasing function. We have hence shown the existence of infinitely many infinite words in $F_4$, none of which contain any $N_i \in \mathcal{N}$ as subwords. The fact that we get such a tree of paths follows from having $6$ options to expand in every step. Hence we have a $7$-regular tree of such paths that embeds into the Cayley graph of $F_4$.
 \end{proof}

 An intuitive description of the above Lemma is that the rays of infinite words in $F_4$ which will eventually become relations in the group $G$, are very sparse in the Cayley graph of $F_4$. In fact, at every point there are $6$ directions which are on a Morse ray and increasing the length. Graphically speaking, it is enough to cut off one branch of the tree $Cay(F_4)$ at every point on a path starting from $1$.

\subsection{Building the quotients}
 
In this subsection we carry our findings about certain paths avoiding certain words in $F_4$ over to a $C(\epsilon, \mu, \rho)$-quotient. We describe how the images of the paths we found in $F_4$ are still geodesics in such a quotient, embed convexly and have no common 'long' parts with any of the relations in $\mathcal{Q}$.

\medskip

The following lemma describes combinatorial properties of the set of all 'sufficiently long' subwords of all relations in a graded small cancellation group. 
 
\begin{lemma}\label{lemma_setN_free}
Assume $\left<G | \mathcal{Q} = \bigcup \mathcal{R}_i \right>$ satisfies a graded small cancellation assumption with $G = \lim_{i \rightarrow \infty} G_i$ with $G_0 = F_4$. Let $\mathcal{N}$ be the following set:
\[\mathcal{N} = \{ w \mid \exists i: \exists R \in \mathcal{R}_i: w \subset R, |w| = \lfloor \sqrt{\mu_i} \sigma_i \rfloor  \mbox{ for } \rho_i < |R| < \sigma_i\}.\]
Then this set satisfies the following: 

\begin{enumerate}[(i)]
 \item \label{lemma_setN_part_1_free} each $N \in \mathcal{N}$  is such that $l(N) \geq 10^4 $, 
 \item \label{lemma_setN_part_2_free} there are at most $8^{n/2}$ words of length $n$ in $\mathcal{N}$,
 \item \label{lemma_setN_part_3_free} let $N_1, N_2 \in \mathcal{N}$ such that $|N_1| \neq |N_2|$. Then $||N_1| -|N_2|| > 10$.
\end{enumerate}

\end{lemma}

\begin{proof} 

\eqref{lemma_setN_part_1_free} We have $\mu_n \rho_n > K\epsilon_n$, and $\epsilon_n\geq 1$ for all $n$ (recall: $K=10^6$), which together give the statement. 

\eqref{lemma_setN_part_2_free} This is the statement of Lemma \ref{lemma_sc_numberPieces}.

\eqref{lemma_setN_part_3_free} We have to check that $\sqrt{\mu_{n+1}}\sigma_{n+1} - \sqrt{\mu_n}\sigma_n > 10$. We use that $\mu_n\sigma_n> K\epsilon_n > \frac{K}{8} \sigma_{n-1}$. This gives us $\sqrt{\mu_{n+1}}\sigma_{n+1} > \mu_{n+1}\sigma_{n+1} > K \epsilon_{n+1} > \frac{K}{8} \sigma_n$ since $\epsilon_{n+1} \geq 8\sigma_n$. Hence we get 
\[\frac{K}{8}\sigma_n - \mu_n\sigma_n > 10\]
\[\sigma_n \left( \frac{K}{8} - \sqrt{\mu_n} \right) > 10. \]
By the graded small cancellation condition $K=10^6$ and $\mu_n\leq 0.01$ and hence $\frac{K}{8}-\sqrt{\mu_n} >10^4$.  By using that $\sigma_n>10^8$ we get that this is always satisfied.
\end{proof}

The set $\mathcal{N}$ from the Lemma above is the set of all 'sufficiently long' subwords of all relations in $\mathcal{Q}$. We show in the next Corollary that the proportion of the subwords $w \in \mathcal{N}$ of a relation $R$ that is included in $\mathcal{Q}$ goes to $0$. 

\begin{cor}
Let $\mathcal{N}$ be as above. For $R \in \mathcal{Q}$ and denote by $k_R = \max\{ |w| \mid w \in \mathcal{N}, w \subset R \}$. Then $\lim_{i \rightarrow \infty} (\max_{R \in \mathcal{R}_i} \{k_R/|R|\}) = 0$.
\end{cor}

\begin{proof}
Follows by construction of $\mathcal{N}$ together with $\mu_i \rightarrow 0$.
\end{proof}

The next lemma says that certain convex subsets of $C(\epsilon, \mu, \rho)$ groups stay convex in appropriate $C(\epsilon_1, \mu_1,\rho_1)$-quotients. We will apply this when passing from $G_i$ to its quotient $G_{i+1}$.

\medskip 

\textbf{Notation:} For better readability, we write $p$ instead of $|p|$ for the length of a geodesic segment $p$ when the context is clear.

\begin{lemma}\label{lemma_KstaysQuasiConvex_free} Let $Y$ be a convex connected subset of the Cayley graph of a $\delta$-hyperbolic group $H=\langle S \mid Q \rangle$ where $Q$ satisfies a $C(\epsilon,\mu,\rho)$ condition and $Y$ has a unique geodesic between any two points. Let $\mathcal{R}$ be elements of $H$, with 

\begin{enumerate}[(a)]
 \item \label{lemma_KstaysQuasiConvex_free_partII} any connected subset of any translate of a fixed geodesic representative of  $R$ which lies in $Y$ is shorter than $l(R)/8$ for every $R \in \mathcal{R}$,
 \item $\mathcal{R}$ is chosen such that $H_1=H/\left<\left<\mathcal{R}\right>\right>$ satisfies a $C(\mu_1,\epsilon_1,\rho_1)$ small cancellation condition with $\epsilon_1 > 10^6, \mu_1 \leq 0.01$.
 \item \label{lemma_KstaysQuasiConvex_free_part3} For any simple geodesic $n$-gon, $n \geq 3$, $T=pa_1\dots a_{n-1}$ in $H$ with $p$ in $Y$, we have $|p| \leq \mu \cdot (\sum_{i=1}^{n-1} |a_i|)$ if $\sum_{i=1}^{n-1} |a_i| > \rho$. 
\end{enumerate}
Then 

\begin{enumerate}[(A)]
\item the group $H_1 = \left<H|R\right>$ is hyperbolic and satisfies a $C(\mu_1, \epsilon_1, \rho_1)$ condition,
 \item the image of $Y$ is a convex subset of $H_1$ that has again a unique geodesic between any two points on it, 
 \item  if $T=p \cdot a_1 \dots a_{n-1}$ is a simple geodesic $n$-gon in $H_1$ but not in $H$ with $p$ in the image of $Y$, then $l(p) \leq \mu_1 \cdot \sum_{i=1}^{n-1}a_i$. In particular, the injectivity radius implies that no such $n$-gon exists in $H_1$ but not in $H$ unless $\sum_{i=1}^{n-1}a_i > \rho$.
\end{enumerate}
\end{lemma}

\begin{proof}
\begin{enumerate}[(A)]

\item Follows from Lemma \ref{lem:graded}.

\item Assume there is a new geodesic $g$ connecting points $x,y$ in $Y$, then $g$ was already a geodesic segment in $H$. Assume that in $H$, the points $x,y$ are connected by a geodesic $q$ which lies on $Y$ and is unique by assumption. Let $R$ be a relation in $\mathcal{R}$. This geodesic bigon spanned by $g,q$ in $H_1$ must contain an $R$-cell $\Pi$ (Figure \ref{fig:geodBigon}). By Lemma \ref{lemma_46Lacunary} there are $\epsilon_1$-subcontiguity diagrams and an inequality \[(\Pi, g, q_g) + (\Pi, q, q_q) > 1-23\mu_1.\]

\begin{figure}
\begin{center}
\labellist
\pinlabel $q$ at 80 140 
\pinlabel $g$ at 50 20
\pinlabel $q_q$ at 140 100
\pinlabel $q_g$ at 140 40
\pinlabel $\Pi$ at 170 70
\endlabellist
\includegraphics[scale=0.6]{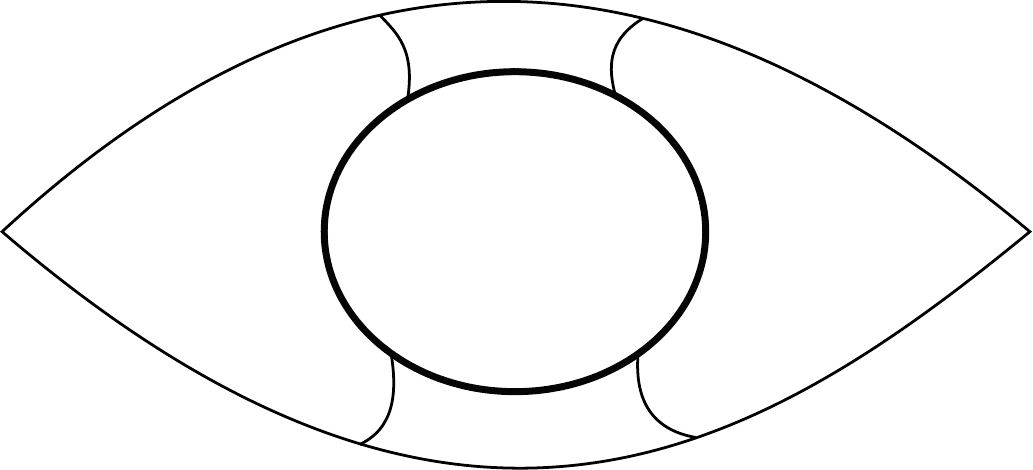}\caption{Geodesic Bigon.}\label{fig:geodBigon}
\end{center}
\end{figure}
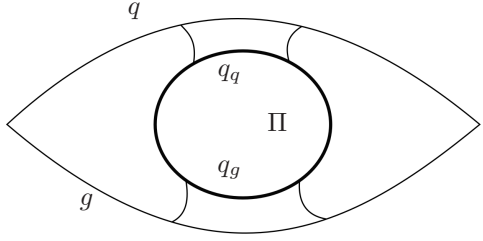

 By assumption, the geodesic $q$ lies in $Y$, hence the intersection of $q$ with $R$ in $H$ has length shorter than $l(R)/8$ by assumption \eqref{lemma_KstaysQuasiConvex_free_partII}. This implies $(\Pi, q, q_q)\leq \frac18 + 2 \frac{\epsilon_1}{\rho_1}$, so \[(\Pi, g, q_g) > 1-23\mu_1-\frac18 - 2 \frac{\epsilon_1}{\rho_1} > 1-0.23 - \frac18 - 2\cdot 10^{-8} \geq \frac12 + 2\cdot 10^{-8} \geq  \frac12 + 2\frac{\epsilon_1}{l(R)}\] because $10^{-8} > \epsilon_1/l(R)$ by Proposition \ref{prop_inequalities}. This contradicts the fact that the new path is geodesic by Lemma \ref{lemma_44Lacunary}. This also implies that the image of $Y$ embeds convexly and has unique geodesics between any two points on it.
\item We discuss the case of a triangle, the general case for an $n$-gon then follows from the triangle inequality. Such a triangle must contain an $R$-cell $\Pi$, otherwise it would already be a triangle in $H$. 

%
%
All sides are geodesics in $H$, so we get $\epsilon_1$-subcontiguity diagrams $\Delta_1, \Delta_2, \Delta_3$ with $|r_i| < \epsilon_1$ and an inequality by Lemma \ref{lemma_46Lacunary} and depicted in Figure \ref{fig:triangle}: 
\begin{equation}(\Pi, p, q_p)+(\Pi, a, q_a)+(\Pi, b, q_b) > 1-23\mu_1.\label{eq_relation_free}\end{equation}

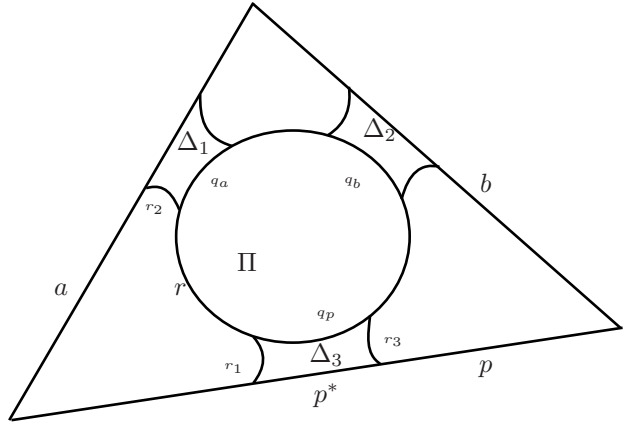
\begin{figure}[h]
\begin{center}
\labellist 
\pinlabel $a$ at 20 50
\pinlabel $b$ at 180 90
\pinlabel $p$ at 180 20
\pinlabel $\Pi$ at 90 60
\pinlabel $p^*$ at 120 10
\pinlabel \tiny{$q_a$} at 80 90
\pinlabel \tiny{$q_b$} at 130 90
\pinlabel \tiny{$q_p$} at 120 40
\pinlabel \tiny{$r_1$} at 85 20
\pinlabel \tiny{$r_2$} at 55 80
\pinlabel \tiny{$r_3$} at 145 30
\pinlabel $r$ at 65 50
\pinlabel $\Delta_1$ at 70 105
\pinlabel $\Delta_2$ at 140 110
\pinlabel $\Delta_3$ at 120 25
\endlabellist

\includegraphics[scale=1]{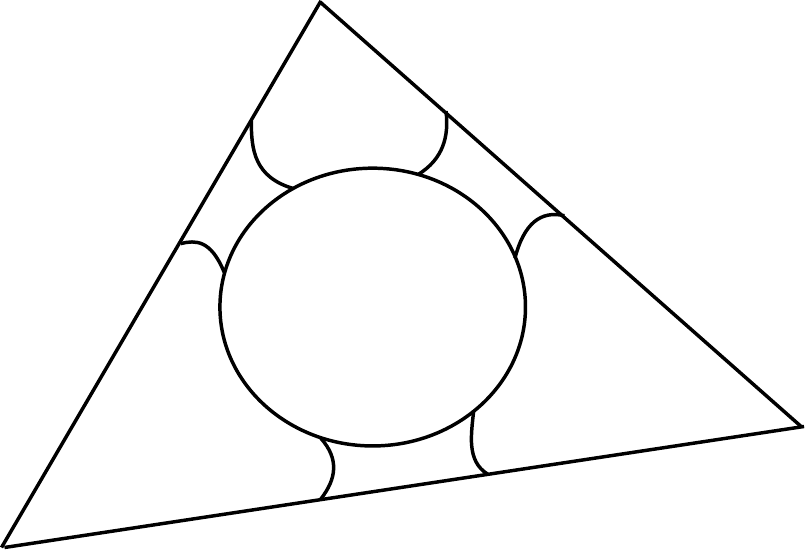}\caption{$\epsilon_1$-contiguity diagrams for a triangle in $H_1$.}\label{fig:triangle}
\end{center}
\end{figure}

By the definition of $\epsilon_1$-contiguity diagrams, none of the $\Delta_i$ contain a further $R$-cell. By hypothesis \eqref{lemma_KstaysQuasiConvex_free_part3}, we get for $p^*$, the segment of $p$ of the adjacent $\epsilon$-contiguity diagram $\Delta_3$ that \begin{equation}p^*<\mu \cdot (r_1+r_3+q_p).\label{eq_pStar}\end{equation} Because the $\epsilon_1$-subcontiguity diagram is a geodesic rectangle, we get further with $\mu \leq 0.01$ that \[q_p \leq r_1+r_3+p^* \leq r_1+r_3+\mu \cdot (r_1+r_3+q_p) \leq 1.01\cdot r_1 + 1.01\cdot r_3 + q_p\cdot \mu\] yielding combined with $|r_i| \leq \epsilon_1$ \[q_p\leq \frac{2.02 \cdot \epsilon_1}{(1-\mu)} \leq \frac{2.02\cdot \epsilon_1}{1-\frac{1}{100}} \leq 3 \epsilon_1\] and so put into \eqref{eq_pStar}  this gives \begin{equation}p^* < \mu \cdot 5 \epsilon_1\label{equation_p_star}\end{equation} which implies by using the triangle inequality in the rectangle $\Delta_3$ that \begin{equation}|q_p| < 7 \mu \epsilon_1.\label{eq:qp}\end{equation}

\medskip

We look at the left corner of this triangle and find a pentagon $P$ with sides $a', r_1, r, r_2, p'$ as shown in Figure~\ref{fig:pentagon}.

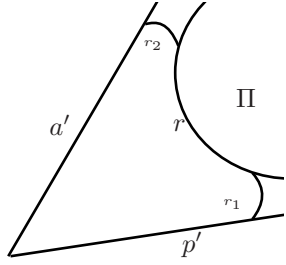
\begin{figure}[h]
\begin{center}
\labellist 
\pinlabel $a'$ at 20 50
\pinlabel $\Pi$ at 90 60
\pinlabel \tiny{$r_1$} at 85 20
\pinlabel \tiny{$r_2$} at 55 80
\pinlabel $r$ at 65 50
\pinlabel $p'$ at 70 5
\endlabellist

\includegraphics[scale=1]{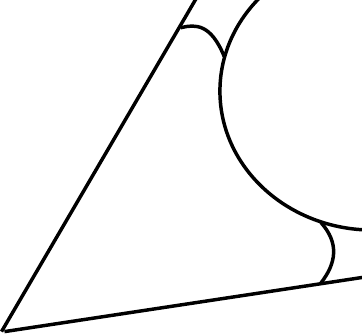}\caption{Corner of the triangle in $H_1$.}\label{fig:pentagon}
\end{center}
\end{figure}

\begin{claim}$P$ does not contain another $R$-cell.\end{claim}
\begin{proof_claim}
Assume it does contain an $R$-cell $\Pi_1$. Then we would have by Lemma \ref{lemma_46Lacunary} for suitable $\epsilon_1$-contiguity diagrams the inequality 
\begin{equation}(\Pi_1, p', q'_p)+(\Pi_1, a', q'_a)+(\Pi_1, r, q_r)+(\Pi_1, r_1, q_{r_1}) + (\Pi_1,r_2, q_{r_2}) > 1-23\mu_1.\label{eq:contPentagon}\end{equation} In particular, there is one $\epsilon_1$-contiguity diagram adjacent to $p'$ with sides $p'^*, r_4, r_5$ and $q_{p'}$. We have from above $|r_i| < \epsilon_1$ and hence $(\Pi_1, r_i,q_{r_i}) < \epsilon_1/\rho_1 <  10^8$. By the small cancellation condition $C(\epsilon_1, \mu_1, \rho_1)$ it is required that $(\Pi_1, r, q_r) < \mu_1$ and by the same argument as above we get $p'^*< \mu \cdot 5 \epsilon_1 < 5 \epsilon_1$. This together would imply with $\epsilon_1/ \rho_1 \leq 10^{-8}$ as in Proposition \ref{prop_inequalities} that \[(\Pi_1, a', q'_a) > 1-23 \mu_1 - 5 \epsilon_1/\rho_1 - 7 \cdot 10^{-8}>\frac23 > \frac12 + 2\epsilon_1\rho_1.\] This contradicts that $a'$ is a geodesic segment by Lemma 4.4 in \cite{osin_lacunary}. So the pentagon $P$ does not contain an $\mathcal{R}$-cell. \end{proof_claim}

\medskip 

We then apply \eqref{lemma_KstaysQuasiConvex_free_part3} to $P$ and get \[p'< \mu (r_1+r+r_2+a').\] The condition $\rho \leq r_1+r+r_2+a'$ follows automatically from the injectivity radius of $H\twoheadrightarrow H_1$. We use that $a' \leq r_2+r+r_1+p'$ by the triangle inequality and hence with $|r_i| \leq \epsilon_1$ 
we get 
\[p' \leq \mu \cdot (2r_1+2r+2r_2 + p') \leq \mu \cdot  ( 4 \epsilon_1 + 2|r| + p')\] leading to \begin{equation}p'\leq \frac{\mu}{1-\mu}(4 \epsilon_1 + 2|r|) \leq 2 \mu \cdot (4 \epsilon_1 + 2 |r|).\label{eq:pdash}\end{equation} On the right side of the triangle in Figure \ref{fig:triangle} we have another pentagon $P'=p''b'\bar{r}r_4r_5$ with side $p''$ part of $p$. The same as above shows 
\begin{equation}p'' \leq 2 \mu \cdot (4 \epsilon_1 + 2 |\bar{r}|).\label{eq:pdashdash}\end{equation}

Equation \eqref{eq:contPentagon} implies $|r| + |\bar{r}|< 23 \mu_1\rho_1$. 

So \[p = p'+p^*+p'' \leq 5 \epsilon_1 + 2\mu \cdot (8 \epsilon_1 + 2(|r|+|\bar{r}|)) \leq 5\epsilon_1 + 16 \mu \epsilon_1 + 2\mu \cdot 46 \mu_1 \rho_1\]
\[\leq \rho_1 (5\cdot 10^{-8} + 16 \cdot 10^{-10} + 92\mu\mu_1) \leq 0.93\mu_1\rho_1.\] 
Now if the triangle $pab$ contains an $R_1$-cell and $|R_1| > \rho_1$ we must have $|a|+|b|+ |p| \geq \rho_1$. With $p < 0.93 \mu_1\rho_1$ we hence must have \[|a| + |b| > \rho_1 - 0.93\mu_1\rho_1 > \rho_1 - 0.0093 \rho_1 > 0.999 \rho_1 > 0.93 \rho_1.\] Hence we obtain $p < \mu_1 \cdot (|a|+|b|)$.
\end{enumerate}
\end{proof}

\subsection{Geometry of graded small cancellation groups} 

Denote by $\mathcal{F}=\lim_{n\rightarrow \infty} F(n)$ the set of all such paths in $F_4$ as described in Lemma \ref{lemma_comb_free}. Denote by $\bar{\mathcal{F}}$ the set of images of those paths in $G=\langle F_4, \mathcal{Q} \rangle$. 



\begin{prop}\label{prop_L}
Let $G=\lim_{i \rightarrow \infty} G_i$ be the group with a graded small cancellation presentation $\left<G | \mathcal{Q} \right>$ with $G_0 = F_4$. Then
\begin{enumerate}[(i)]
 \item $\tilde{\mathcal{F}}$ is a convex subset of $Cay(G)$,
 \item any geodesic triangle $pab$ with $p$ a segment of $\tilde{\mathcal{F}}$ satisfies: if $\rho_{i+1} > |a| + |b| > \rho_i$, then $|p| < \mu_i (|a|+|b|)$.
\end{enumerate}
\end{prop}

\begin{proof}
We show by induction that $M$ satisfies the conditions \eqref{lemma_KstaysQuasiConvex_free_partII}-\eqref{lemma_KstaysQuasiConvex_free_part3} of Lemma \ref{lemma_KstaysQuasiConvex_free}. For $i=0$ we have $H_0=F_4$ and hence nothing to show. Now assume $\tilde{\mathcal{F}}$ satisfies these conditions in $H_i$. Then by Lemma \ref{lemma_KstaysQuasiConvex_free} it satisfies (b) and (c) in $H_{i+1}$. By construction of the set $\tilde{\mathcal{F}}$ it also satisfies (a), and hence (a)-(c) in $H_{i+1}$ and hence (i) and (ii).
\end{proof}

\begin{lemma}\label{lem:nohexagons}
Let $A$ be a simple geodesic hexagon in $G$ with $p$ a side that is a segment of $\tilde{\mathcal{F}}$. Further, assume that we have three short sides $s_1, s_2, s_3$ with $|s_i| < |p|/4$ for all $i=1,2,3$. Then $|s_i|=0$ for all $i=1,2,3$, i.e.\;the hexagon is a triangle, and if the other two sides $a,b$ satisfy $|a|+|b| > \rho_i$ then $|p|< \mu_i \cdot  (|a|+|b|)$.
\end{lemma}

\begin{proof}
Assume $|s_i|>0$. Let $s_1$ be adjacent to $p$. Then there exists a triangle $s_1pk$. Since $G$ is a direct limit of groups $G_i$, this triangle must be a triangle in some $G_i$ but not in $G_{i-1}$. Hence by Proposition \ref{prop_L} we obtain \[p < \mu_i \cdot (s_1+k) < \mu_i \cdot (s_1+s_1+p) = 2 \mu_i s_1 + \mu_i p.\]
This leads to 
\[p < \frac{2}{1-\mu_i} s_1 < \frac{p}{2(1-\mu_i)}\]
\[2(1-\mu_i) < 1,\] a contradiction since $\mu_i \leq 0.01$.
\end{proof}

\begin{cor}\label{cor:triangle}
Let $T=abp$ be a simple geodesic triangle in $G$ with $p \subset \mathcal{F}$ and $\rho_i < |a|+|b|$. Then there exists a point $x$ on $T$ with $d(x,p) \geq \frac{|p|}{2\pi \mu_i}$.
\end{cor}

\begin{proof}
We subscribe a circle around the triangle. Then by Proposition \ref{prop_L} the circumference $C$ must be at least $C > |p|/\mu_i$ since $|a|+ |b| > |p|/\mu_i$. The side $p$ has length at most the diameter $d$ of the surrounding triangle. Hence we have with $C = \pi \cdot d$ that $\pi \cdot d > |p|/\mu_i$. The point $x$ on the triangle which is the furthest from $p$ has distance at least the radius of the circle and hence $d(x,p) > |p|/(2\pi \mu_i)$.
\end{proof}

\begin{prop}\label{prop:nodoublelines}
Fix $\{d_i\}, \omega$ and hence an asymptotic cone. Let $q_i$ be a sequence of geodesic segments such that $\gamma= \lim_\omega q_i$ is contained in the ultralimit $\bar{\mathcal{F}}$ of $\tilde{\mathcal{F}}$. Then either $\gamma$ is a point or all but finitely many  $q_i$ are contained in $\tilde{\mathcal{F}}$.
\end{prop}

We recall a notation used in the following proofs: Let $x,y$ be points in a geodesic metric space. Then by $[x,y]$ we denote a geodesic segment starting at $x$ and ending at $y$. In the case of multiple geodesic segments between these two points, we mean any of these geodesic segments, not to be specified.

\begin{proof}
Assume the limit of the $q_i$ is a geodesic segment $\tilde{q}$ inside the limit $\bar{\mathcal{F}}$ of $\tilde{\mathcal{F}}$. If not all $q_i$ lie on $\tilde{\mathcal{F}}$, then there exists an index $i$ such that $q_i$ is not fully contained in $\tilde{\mathcal{F}}$ in $Cay(G)$. Because the injectivity radius of the maps $G_i \twoheadrightarrow G_{i+1}$ goes to infinity we can assume we are in the group $H_i$. Let $q$ be a geodesic segment on $\mathcal{F}$ such that the limit of $q$ is also $\tilde{q}$. Let $x$ be a point on $q_i$ such that $x \notin \tilde{\mathcal{F}}$ with maximal distance to $\tilde{\mathcal{F}}$. Then there exists a geodesic triangle $[q_-,x], [x,q_+], [q_+, q_-]$. Now by Corollary \ref{cor:triangle} the point $x$ must be at distance at least \[d(x,[q_+,q_-]) \geq \frac{|[q_+,q_-]|}{2\pi \mu_i}.\] Together with $lim_{i \rightarrow \infty} \mu_i = 0$ we see that there cannot be infinitely many $q_i$ which lie outside $\tilde{\mathcal{F}}$ but the limit $\lim_\omega q_i$ is inside the limit of $\tilde{\mathcal{F}}$.
\end{proof}

\begin{theorem}\label{thm:morse}
Let $G = \lim_{i \rightarrow \infty} G_i$ be a group which satisfies a graded small cancellation condition with $G_0 = F_4$. Then $G$ contains an isometrically embedded $7$-regular tree $T$, where every bi-infinite path is a Morse geodesic.
\end{theorem}

\begin{proof}
We show that the ultralimit $\bar{\mathcal{F}}$ of $\tilde{\mathcal{F}}$ is a transversal tree in any asymptotic cone. Let $M$ be a bi-infinite path in $\tilde{\mathcal{F}}$ and let $q$ be its ultralimit in an asymptotic cone $\mathcal{A}$. We have to show 
\begin{enumerate}
 \item every point of $q$ is a cut-point in $\mathcal{A}$,
 \item every point of $q$ separates $q$ into two halves which lie in two different connected components.
\end{enumerate}

We notice that condition (2) implies (1). Assume there exists a point $x$ on $q$ which does not satisfy (2). Now assume that removing $x$ from $q$ does not break $q$ into two halves which lie in different connected components. Then there must exist a path $\gamma$ from a point $t$ on $q$ to a point $s$ on $q$, where $t,s$ lie on different sides of $x$. An asymptotic cone is a geodesic metric space, so in particular there is a point $y$ on $\gamma$ and geodesics $v_1=[t,y], v_2=[y,s]$. The path $q$ is a geodesic, we have a geodesic triangle $T=v_1v_2p$ where $p=[t,s]$ on $q$. The triangle $T$ must be the ultralimit of geodesic hexagons in $G$. By Lemma \ref{lem:nohexagons} it must be the limit of geodesic triangles in $G$. By Proposition \ref{prop:nodoublelines} the base line of these triangles must be contained in $\tilde{\mathcal{F}}$. However Proposition \ref{prop_L} states that the limit of such triangles with the base line in $\tilde{\mathcal{F}}$ is never a simple geodesic triangle in any asymptotic cone. 

Hence the point $x$ on $q$ must be a cut-point and be such, that the resulting two halves of $q$ lie in different connected components. This means that $\bar{\mathcal{F}}$ is a transversal tree in the asymptotic cone and hence by Proposition~\ref{prop_transversal} this implies that any bi-infinite path on it is a Morse geodesic in $G$.

\end{proof}

\begin{cor}
There exists a group $G$ such that every $g$ in $G$ has finite order and hence $G$ contains no Morse elements. But $G$ contains Morse geodesics. 
\end{cor}

\begin{proof}
In \cite{osin_lacunary} the authors construct infinite groups with unbounded torsion that satisfy the graded small cancellation condition. Hence these groups contain Morse geodesics by Theorem \ref{thm:morse} but on the other hand torsion elements can never be Morse elements. 
\end{proof}

\renewcommand{\bibname}{References}
\bibliography{bibliography.bib}        

\end{document}